\documentclass{amsart}

\usepackage{amsmath,amssymb,verbatim}
\usepackage{amsthm}
\usepackage{enumerate}

\newtheorem{theorem}{Theorem}[section]
\newtheorem{fact}[theorem]{Fact}
\newtheorem{lemma}[theorem]{Lemma}
\newtheorem{corollary}[theorem]{Corollary}
\newtheorem{proposition}[theorem]{Proposition}

\theoremstyle{definition}
\newtheorem{definition}[theorem]{Definition}
\newtheorem{example}[theorem]{Example}
\newtheorem{remark}[theorem]{Remark}

\newcommand{\abar}{\bar{a}}
\newcommand{\bbar}{\bar{b}}
\newcommand{\cbar}{\bar{c}}
\newcommand{\xbar}{\bar{x}}
\newcommand{\ybar}{\bar{y}}
\newcommand{\zbar}{\bar{z}}

\def\seq{\subseteq}
\def\nv{\text{-}}
\def\inv{^{\text{-}1}}

\def\smd{\raisebox{.4pt}{\textrm{\scriptsize{~\!$\triangle$\!~}}}}
\newcommand{\claim}[1]{\hfill$\dashv_{\text{\scriptsize{Claim {#1}}}}$}
\def\Stab{\operatorname{Stab}}
\def\N{\mathbb{N}}

\def\E{\mathbb{E}}
\def\Q{\mathbb{Q}}

\def\cL{\mathcal{L}}
\def\cU{\mathcal{U}}
\def\cP{\mathcal{P}}
\def\cF{\mathcal{F}}
\def\cB{\mathcal{B}}

\def\cI{\mathcal{I}}
\def\cS{\mathcal{S}}

\def\tp{\operatorname{tp}}
\def\opp{\operatorname{opp}}
\def\Def{\operatorname{Def}}

\title{Pseudofinite groups and VC-dimension}

\author{Gabriel Conant and Anand Pillay}
\thanks{The authors were supported by NSF grants DMS-1855503 (Conant) and DMS-1360702, DMS-1665035, DMS-1760212 (Pillay). Much of the work in this paper was done during the Model Theory, Combinatorics and Valued Fields trimester program at Institut Henri Poincar\'{e}. We thank IHP for their hospitality. The first author  thanks Artem Chernikov for helpful conversations.}

\address{Department of Mathematics\\
University of Notre Dame\\
Notre Dame, IN, 46656, USA}

\email{gconant@nd.edu}

\address{Department of Mathematics\\
University of Notre Dame\\
Notre Dame, IN, 46656, USA}

\email{apillay@nd.edu}

\date{May 21, 2020}

\begin{document}

\begin{abstract}
We develop ``local NIP group theory" in the context of pseudofinite groups. In particular, given a sufficiently saturated pseudofinite structure $G$ expanding a group, and  left invariant NIP formula $\delta(x;\ybar)$, we prove various aspects of ``local \emph{fsg}" for the right-stratified formula $\delta^r(x;\ybar,u):=\delta(x\cdot u;\ybar)$. This includes a $\delta^r$-type-definable connected component, uniqueness of the pseudofinite counting measure as a left-invariant measure on $\delta^r$-formulas, and generic compact domination for $\delta^r$-definable sets. 
\end{abstract}

\maketitle

\section{Introduction}
\setcounter{theorem}{0}
\numberwithin{theorem}{section}

\let\thefootnote\relax\footnotetext{\emph{2010 MSC}: 03C45, 03C20; \emph{keywords}: pseudofinite groups, VC-dimension, NIP formulas}
One of the more remarkable aspects of stable group theory is the ability to formulate useful abstract notions of tools from algebra, combinatorics, and topological dynamics. For example, given a group $G$ definable in a (sufficiently saturated) model of a stable theory, one has at hand abstract versions of connected components, stabilizers, generic points, and invariant probability measures on definable sets (leading to the notion of \emph{definable amenability}).  As the entire field of model theory began moving outward from stability, so did the model theoretic study of groups, leading to a large body of work on groups definable in \emph{simple} and \emph{NIP} theories. In the case of NIP theories, connected components and invariant measures remain powerful tools for studying definable groups (e.g \cite{ChSi}, \cite{HPP}, \cite{HP}, \cite{HPS}). 

Another important aspect of stability theory is that it can be applied locally. Indeed, many of the tools related to nonforking and the geometry of definable sets remain valid when one works around a single stable formula $\phi(\xbar;\ybar)$ (see, e.g., \cite{HrPiGLF}). This is quite useful for applications to other areas of mathematics, as one would like to prove results about algebraic or combinatorial objects exhibiting good behavior related to stability and omitting half-graphs, but also have the freedom to work in an environment which is not stable (e.g. a nonstandard model  of set theory). 

In contrast, the local study of NIP formulas is still work in progress, especially in the setting of groups. The goal of this paper is to examine NIP formulas in the context of pseudofinite groups. We will find that, in pseudofinite groups, NIP formulas exhibit many properties found in NIP groups with \emph{finitely satisfiable generics} and \emph{generically stable measures}. In fact, the results we obtain here could be formulated with the pseudofiniteness assumption replaced by a suitably local, albeit slightly cumbersome, assumption of generic stability for an invariant measure with respect to an NIP formula (see Remark \ref{rem:localFSG}). On the other hand, it is quite reasonable to focus on pseudofinite groups, as this is a natural settings for applications to finite combinatorics and combinatorial number theory. Indeed, the work in this paper was originally motivated by generalizing the regularity lemma for stable subsets of finite groups from \cite{CPT} (with C. Terry), to the NIP setting. In \cite{CPTNIP} (also with Terry), we use the work done here to obtain such a generalization.

Our setting is as follows. We work with a sufficiently saturated elementary extension $G$ of an ultraproduct of expansions of finite groups (in some fixed language $\cL$ expanding the language of groups), and let $\mu$ denote the pseudofinite counting measure. We also fix a formula $\delta(x;\ybar)$, possibly with parameters, which is  invariant in the sense that any left translate of an instance of $\delta$ is again an instance of $\delta$ (the canonical example of an invariant formula is something of the form $\phi(y\cdot x)$, where $\phi(x)$ is any formula). In order to prove our main results, it will be necessary to work mostly around the ``right-stratified" formula $\delta^r(x;\ybar,u)$, which we define to be $\delta(x\cdot u;\ybar)$. The following theorem summarizes the main results of the paper.

\begin{theorem}\label{thm:main}
Let $G$ and $\delta(x;\ybar)$ be as above, and assume $\delta(x;\ybar)$ is NIP. 
\begin{enumerate}[$(a)$]
\item \textnormal{(Generic types)} Given a $\delta^r$-formula $\phi(x)$, the following are equivalent:
\begin{enumerate}[$(i)$]
\item $\phi(x)$ is left generic;
\item $\phi(x)$ is right generic;
\item $\mu(\phi(x))>0$.
\end{enumerate}
In particular, global generic $\delta$-types and  global generic $\delta^r$-types exist.
\vspace{3pt}
\item \textnormal{(Local $G^{00}$)} Let $G^{00}_{\delta^r}$ denote the intersection of all $\delta^r$-type-definable bounded-index subgroups of $G$. Then:
\begin{enumerate}[$(i)$]
\item $G^{00}_{\delta^r}$ is normal and $\delta^r$-type-definable of bounded index. 
\item $G^{00}_{\delta^r}$ is the intersection of all stabilizer subgroups of the form $\Stab_\mu(\phi(x)):=\{g\in G:\mu(\phi(g\inv x)\smd\phi(x))=0\}$, where $\phi(x)$ is a $\delta^r$-formula.
\item $G^{00}_{\delta^r}$ is the intersection of all stabilizer subgroups of the form $\Stab(p):=\{g\in G:gp=p\}$, where $p$ is a generic $\delta$-type over $G$. 
\item $G^{00}_{\delta^r}=\Stab(p)$ for any generic $\delta^r$-type $p$ over $G$.
\end{enumerate}
\vspace{3pt}
\item \textnormal{(Local $G^0$)} Let $G^0_{\delta^r}$ denote the intersection of all $\delta^r$-definable finite-index subgroups of $G$. Then $G^0_{\delta^r}$ is normal and $\delta^r$-type-definable of bounded index. Moreover, $G^0_{\delta^r}/G^{00}_{\delta^r}$ is the connected component of the identity in $G/G^{00}_{\delta^r}$.
\vspace{3pt}
\item \textnormal{(Generic compact domination)} Given a $\delta^r$-formula $\phi(x)$, define $E_{\phi(x)}\seq G/G^{00}_{\delta^r}$ to be the (closed) set of $C\in G/G^{00}_{\delta^r}$ such that $p\models C\cap\phi(x)$ and $q\models C\cap\neg\phi(x)$ for some generic $\delta^r$-types $p,q$ over $G$. Then $E_{\phi(x)}$ has Haar measure $0$.
\vspace{3pt}
\item \textnormal{(Uniqueness of measure)} The measure $\mu$ is the unique left-invariant finitely additive probability measure on the Boolean algebra $\cB$ of $\delta^r$-formulas. Moreover, $\mu$ is the unique left $G^{00}_{\delta^r}$-invariant finitely additive probability measure on $\cB$, which lifts the Haar measure on $G/G^{00}_{\delta^r}$.
\end{enumerate}
\end{theorem}

\begin{remark}
The reader familiar with \cite{ChSi}, \cite{HPP}, \cite{HP}, and \cite{HPS} will notice strong similarities between the theorem above and properties of groups definable in NIP theories, especially those with \emph{finitely satisfiable generics} (see Remark \ref{rem:localFSG}). Therefore, it is worth emphasizing that although these sources provide an invaluable guide for the proof of Theorem \ref{thm:main} (as detailed below), the results above are not simply obtained by direct translation from the global NIP setting. In particular, in addition to the usual obstacles when working locally, the formula $\delta^r(x;\ybar)$ need not be NIP (see Example \ref{ex:bad}), and so we are instead forced to work with certain families of (non-uniformly) NIP formulas given by instances of $\delta^r$. For this reason, many proof techniques from the above sources must either be redone from scratch, or replaced with a more combinatorial method that does not rely on model-theoretic tools. 
\end{remark}

\begin{remark}\label{rem:localFSG}
The essential use of pseudofiniteness is in our work is in deriving local analogues of definability and finite satisfiability of the pseudofinite counting measure, with respect to certain families of NIP formulas. Indeed, our results hold in a more general setting of ``local fsg" for NIP formulas, and so we take the opportunity here to make this explicit. Recall  that a group $G$, definable in a saturated model of an NIP theory, has \emph{finitely satisfiable generics} if and only if it admits a left-invariant \emph{Keisler measure} (i.e., a finitely additive probability measure on definable sets), which is \emph{generically stable} (i.e., definable and finitely satisfiable) over a small model (see \cite[Proposition 8.33]{Sibook}). Moreover, this measure is right-invariant and unique. A suitable localization of this behavior is as follows. 

Suppose $G$ is a saturated expansion of a group and $\delta(x;\ybar)$ is an invariant NIP formula. Call a formula $\theta(x;\ybar)$ a \emph{stabilizing formula} if it is of the form $\phi(x\cdot y)$, $\phi(y\cdot x)$, $\phi(y_1\cdot x)\smd\phi(y_2\cdot x)$, or $\phi(x\cdot y_1)\smd\phi(x\cdot y_2)$ for some $\delta^r$-formula $\phi(x)$. For the purposes of this remark, we say that $\delta(x;\ybar)$ is \emph{fsg} if  there is a left and right invariant measure $\mu$ on the Boolean algebra of $\delta^r$-definable sets such that, for any stabilizing formula $\theta(x;y_1,y_2)$ and any $\epsilon>0$, the following holds:
\begin{enumerate}[$(i)$]
\item \emph{(local finite satisfiability)} there is $F\seq G$ finite such that, for any $b_1,b_2\in G$, if $\mu(\theta(x;b_1,b_2))>\epsilon$ then $\theta(x;b_1,b_2)$ is realized in $F$, and
\item \emph{(local definability)} the set $\{(b_1,b_2)\in G^2:\mu(\theta(x;b_1,b_2))\leq\epsilon\}$ is $\theta^{\opp}$-type-definable.
\end{enumerate}
In this case, the measure $\mu$ satisfies the key properties demonstrated for pseudofinite counting measures in Section \ref{sec:pre}. Altogether, if $G$ is as above and $\delta(x;\ybar)$ is invariant, NIP, and fsg, then Theorem \ref{thm:main} holds, with $\mu$ in place of the pseudofinite counting measure. Although we do not pursue any examples of this behavior beyond NIP formulas in pseudofinite groups, it is worth pointing out that if $G$ is an \emph{fsg} group definable in an NIP theory, then any invariant formula $\delta(x;\ybar)$ in $G$ is (locally) fsg. Indeed, the unique left-invariant Keisler measure on $G$ is generically stable, and thus satisfies $(i)$ and $(ii)$ with respect to \emph{any} formula $\theta(x;\ybar)$ (for $(ii)$, $\theta^{\opp}$-type-definability follows from \emph{finite approximability} of the measure, as described in \cite[Theorem 7.29$(ii)$]{Sibook}).  Altogether, our work shows that the properties of NIP fsg groups can be localized around objects defined by controlled instances of a particular formula.
\end{remark}

\begin{remark}
The initial motivation for the work in this paper was toward the subject of arithmetical regularity in finite groups (first developed by Green \cite{GreenSLAG} in the abelian setting). In analogy to the various strengthened regularity lemmas for graphs definable in model theoretically tame (e.g. stable or NIP) contexts, there has been a flurry of recent interest in strengthened arithmetic regularity lemmas for similarly ``tame" subsets of finite groups. In \cite{CPTNIP} (also with C. Terry), we use Theorem \ref{thm:main}, and generic compact domination in particular, to prove arithmetic regularity lemmas for ``VC-sets" in  finite groups (i.e., sets whose family of left translates has absolutely bounded VC-dimension).  Further detail is given at the end of this section.
\end{remark}

We now give a brief summary of the paper and compare the various aspects of the above theorem to previous work on groups definable in NIP theories. 

Section \ref{sec:pre} contains preliminary observations on the pseudofinite setting above. In Section \ref{sec:pre2}, we show that the VC-theorem transfers naturally to pseudofinite structures and, as a consequence, the pseudofinite counting measure is definable and finitely satisfiable when restricted to NIP formulas. 
We also prove a uniqueness result for measures (Theorem \ref{thm:unique}) which, in conjunction with Proposition \ref{prop:deltarfacts}$(a)$, yields the first statement in Theorem \ref{thm:main}$(e)$.

In Section \ref{sec:G00}, we construct $G^{00}_{\delta^r}$ by hand using stabilizers of formulas and generic types. This section contains proofs of parts $(a)$ and $(b)$ of Theorem \ref{thm:main} (see Proposition \ref{prop:deltarfacts} and Theorem \ref{thm:G00}, respectively). These results are local versions, in this pseudofinite setting, of previous results on \emph{fsg} groups in NIP theories. Namely, if $G$ is definable in an NIP theory, then the type-definable connected component $G^{00}$ exists, and is the intersection of all type-definable bounded-index subgroups of $G$. If $G$ is also \emph{fsg}, then generic types exist and left and right genericity coincides. Moreover, in this case there is a left-invariant Keisler measure $\mu$ on $G$, which is definable and finitely satisfiable in some (any) small model. See \cite{HP}, \cite{HPS}.

In Section \ref{sec:G0}, we first ``localize" the standard logic topology on $G/\Gamma$, where $\Gamma$ is type-definable of bounded index, and use this to prove part $(c)$ of Theorem \ref{thm:main} (see Corollary \ref{cor:G00connect}). This is a local analog of the fact that, for a group $G$ definable in an NIP theory, the definable connected component $G^0$ exists, and is the intersection of all definable finite-index subgroups of $G$. Moreover, in this case, $G^0/G^{00}$ is the connected component of the identity in $G/G^{00}$.  

We prove parts $(d)$ and $(e)$ of Theorem \ref{thm:main} in Section \ref{sec:GCD} (see Theorem \ref{thm:GCD} and Corollary \ref{cor:GCD}). These are local analogs of the fact that an \emph{fsg} group $G$ definable in an NIP theory satisfies generic compact domination, and there is a unique left $G^{00}$-invariant Keisler measure lifting the Haar measure on $G/G^{00}$. These results first appeared in \cite{HPS}, although with some errors, and a correct proof was given by Simon in \cite{SimGCD}. Our proofs  rely heavily on results of Simon from \cite{SimRC} and \cite{SimGCD}, and also involve local versions of several proofs in the work of Chernikov and Simon on definably amenable NIP groups \cite{ChSi}.

The study of generic compact domination (and its stronger relative ``compact domination") originates from the Pillay conjectures on groups definable in o-minimal theories (see \cite{HPP}, \cite{PilCLG}). It is rather remarkable that generic compact domination describes, in an infinite setting, the underlying qualitative mechanics of regularity lemmas in model theoretically tame environments, especially arithmetic regularity in the context of finite groups. In particular, given a finite group $G$ and a suitably tame (e.g. stable or NIP) set $A\seq G$, the strongest kind of arithmetic regularity lemma would produce a normal subgroup $H$, whose index is uniformly bounded in some way, such that almost all cosets of $H$ are ``strongly regular for $A$", i.e. are almost entirely contained in $A$ or almost entirely disjoint from $A$ (see \cite{CPT} for a precise account in the stable context). In the above setting of pseudofinite groups (which arise when proving regularity for finite groups via ultraproducts), generic compact domination says that if $A\seq G$ is suitably NIP (e.g. defined by a $\delta^r$-formula as above), almost all cosets of $G^{00}$ are strongly regular for $A$. This generalizes the behavior of subsets of groups definable by stable formulas, in which case $G^{00}=G^0$, and \emph{all} cosets of $G^0$ are strongly regular for $A$, leading to a structural description of $A$ as a union of cosets of $G^0$ (see \cite{CPT}). In the NIP setting, the use of generic compact domination, to deduce results about finite groups using ultraproducts, requires a great deal of further work, which we carry out in \cite{CPTNIP} with C. Terry.

\section{Preliminaries} \label{sec:pre}

\subsection{Set systems and VC-dimension}\label{sec:pre1}
In this section, we briefly state the basic definitions and main results on VC-dimension. Further details can be found in \cite{Sibook}, for example.

A \emph{set system} is a pair $(X,\cS)$ where $X$ is a set and $\cS\seq\cP(X)$. 

\begin{definition}
Let $(X,\cS)$ be a set system.
\begin{enumerate}
\item The \textbf{shatter function} of $(X,\cS)$ is $\pi_{(X,\cS)}\colon\N\to\N$ such that 
\[
\pi_{(X,\cS)}(n)=\max\{|A\cap\cS|:A\seq X,~|A|=n\},
\]
where, given $A\seq X$, $A\cap\cS=\{A\cap Y:Y\in\cS\}$.
\item The \textbf{VC-dimension} of $(X,\cS)$ is 
\[
\sup\{n\in\N:\pi_{(X,\cS)}(n)=2^n\}\in\N\cup\{\infty\}.
\]
\end{enumerate}
\end{definition}

\begin{fact}[Sauer-Shelah Lemma]
For all $k\geq 1$ there is $c=c(k)$ such that, if $(X,\cS)$ is a set system of VC-dimension  $k$, then $\pi_{(X,S)}(n)\leq cn^k$ for all $n\geq 0$.
\end{fact}

Given a set $X$, a unary relation $U$ on $X$, and a tuple $(a_1,\ldots,a_n)\in X^n$, define
\[
\E(a_1,\ldots,a_n;U):=\textstyle\frac{1}{n}|\{i\in [n]:U(a_i)\text{ holds}\}|.
\]
For finite sets $X$, we let $\mu_X$ denote the normalized counting measure on $X$.

\begin{fact}[VC-Theorem]
For any $k\geq 1$ and $\epsilon>0$ there is $r=r(k,\epsilon)$ such that the following holds. Suppose $X$ is a finite set and $(X,\cS)$ is a set system with VC-dimension strictly less than $k$. Then there are (not necessarily distinct) $x_1,\ldots,x_r\in X$ such that $|\mu_X(Y)-\E(x_1,\ldots,x_r;Y)|<\epsilon$ for any $Y\in\cS$.
\end{fact}

The sequence $(x_1,\ldots,x_r)$ in the VC-Theorem is often called an \emph{$\epsilon$-approximation} for the set system $(X,\cS)$, and the set $\{x_1,\ldots,x_r\}$ is an \emph{$\epsilon$-net} for $(X,\cS)$.

\begin{remark}\label{rem:VCbound}
Several of the following results will yields various bounds, depending on some $k$ and $\epsilon$, which are explicit in terms of $r(k,\epsilon)$ in the VC-Theorem. So it is worth noting that $r(k,\epsilon)$ is $O(k\epsilon^{\nv 2}\log(\epsilon\inv))$ (see \cite{HaWe}, \cite{KoPaWo}). 
\end{remark}

\subsection{NIP formulas}\label{sec:pre1.5}

Let $\cL$ be a first-order language, and let $M$ be a fixed sufficiently saturated $\cL$-structure. By convention, a \emph{formula} allows parameters from $M$. We will use \emph{$\cL$-formula} to specify formulas with no extra parameters.

\begin{definition}
Let $\phi(\xbar;\ybar)$ be a formula.
\begin{enumerate}
\item An \textbf{instance of $\phi(\xbar;\ybar)$} is a formula  $\phi(\xbar;\bbar)$ or $\neg\phi(\xbar;\bbar)$, where $\bbar\in M^{\ybar}$. 
\item A \textbf{$\phi$-formula} is a finite Boolean combination of instances of $\phi(\xbar;\ybar)$.
\item A set $X\seq M^{\xbar}$ is \textbf{$\phi$-definable} if it is defined by a $\phi$-formula.  
\item A set $X\seq M^{\xbar}$ is \textbf{$\phi$-type-definable} if it is defined by an intersection of boundedly many $\phi$-formulas. 
\item Let $\phi^{\opp}(\ybar;\xbar)$ denote $\phi(\xbar;\ybar)$.  
\end{enumerate}
\end{definition}

\begin{definition}
Given $k\geq 1$, a formula $\phi(\xbar,\ybar)$ is \textbf{$k$-NIP} if there do not exist sequences $(\abar_i)_{i\in[k]}$ in $M^{\xbar}$ and $(\bbar_X)_{X\seq[k]}$ in $M^{\ybar}$ such that $M\models\phi(\abar_i,\bbar_X)$ if and only if $i\in X$. A formula $\phi(\xbar;\ybar)$ is \textbf{NIP} if it is $k$-NIP for some $k\geq 1$.
\end{definition}

\begin{remark}
A formula $\phi(\xbar;\ybar)$ is $k$-NIP if and only if  the set system $(M^{\xbar},\{\phi(M^{\xbar};\bbar):\bbar\in M^{\ybar}\})$ has VC-dimension at most $k-1$.
\end{remark}

Next we prove a local version of an important result from the study of NIP theories, which often is referred to by the slogan: ``Keisler measures in NIP theories are approximated by types". The proof uses a more general version of the VC Theorem  suitable for set systems on infinite ground sets. 

\begin{proposition}\label{prop:VCmt}
Suppose $\phi(\xbar;\ybar)$ is an NIP formula, and $\nu$ is a finitely-additive probability measure on the Boolean algebra of $\phi$-formulas. Let $M_0\prec M$ be a small elementary substructure containing all parameters used in $\phi(\xbar;\ybar)$. Then for any $\epsilon>0$, there are (not necessarily distinct) $\abar_1,\ldots,\abar_n\in M^{\xbar}$ such that, for all $\bbar\in M_0^{\ybar}$, 
\[
|\nu(\phi(\xbar;\bbar))-\E(\abar_1,\ldots,\abar_n;\phi(\xbar;\bbar))|\leq\epsilon.
\]   
\end{proposition}
\begin{proof}
We follow Section 4 of  \cite{HP}, and \cite[Lemma 4.8]{HP} in particular. Note that $\nu$ extends to a regular Borel probability measure on $S_{\phi}(M)$, which we also denote $\nu$. Applying the VC Theorem$^*$ \cite[p. 1025]{HP}, as in \cite[Lemma 4.8]{HP}, we obtain (not necessarily distinct) $p_1,\ldots,p_n \in S_\phi(M)$ such that, for any $\bbar\in M^{\xbar}$,
\[
\left|\nu(\phi(\xbar;\bbar))-\E(p_1,\ldots,p_n;\phi(\xbar;\bbar))\right|\leq\epsilon,
\] 
where in the righthand expression we identify $\phi(\xbar;\bbar)$ with a clopen set in $S_\phi(M)$. (In fact, the set of such tuples $(p_1,\ldots,p_n)$ has positive measure with respect to the product $\nu^n$.) Now, for $1\leq i\leq n$, choose a realization $\abar_i$ of $p_i|_{M_0}$. If $\bbar\in M_0^{\ybar}$, then $\phi(\xbar;\bbar)\in p_i$ if and only if $M\models\phi(\abar_i;\bbar)$, and so we have the desired conclusion. 
\end{proof}

\subsection{NIP formulas in pseudofinite structures}\label{sec:pre2}

Let $\cL$ be a first-order language. In preparation for working with pseudofinite $\cL$-structures, we expand $\cL$ to a language $\cL^+$ containing a new sort $\cI$, on which there is a binary relation $<$ and a binary function $d(x,y)$. In any finite $\cL$-structure, we interpret $\cI$ as $[0,1]$ and $d(x,y)$ as the standard distance on $[0,1]$ (we will also write $|x-y|$ for $d(x,y)$). For every $\cL$-formula $\phi(\xbar;\ybar)$, we add to $\cL^+$ a $\ybar$-ary function symbol $\mu_\phi(\ybar)$ into $\cI$. In any finite $\cL$-structure $A$, $\mu_\phi(\ybar)$ is interpreted as $\mu_A(\phi(A^{\xbar},\ybar))$.

Let $M$ be a fixed, sufficiently saturated elementary extension of an ultraproduct of finite $\cL^+$-structures (which are canonically expanded from $\cL$-structures as described above). By convention, formulas will always be in the language $\cL$.  We let $\mu$ denote the pseudofinite counting measure on $M$. Specifically, given an $\cL$-formula $\phi(\xbar,\ybar)$ and $\bbar\in M^{\ybar}$, $\mu(\phi(\xbar,\bbar))$ is defined as the standard part of $\mu_\phi(\bbar)$. It is routine to verify that $\mu$ is a finitely additive probability measure on (powers of) $M$.

This section contains several corollaries of the VC-Theorem for pseudofinite structures.
Roughly speaking, the VC-Theorem says that, restricted to set systems of finite VC-dimension, counting measures on finite sets are approximated by averages of points. We now observe that this immediately implies the same statement for the pseudofinite counting measure on $M$. 

\begin{corollary}\label{cor:VCpf}
For any $k\geq 1$ and $\epsilon>0$, there is $r=r(k,\epsilon)$ such that the following holds. Suppose $\phi(\xbar;\ybar)$ is a $k$-NIP formula. Then there are (not necessarily distinct) $\abar_1,\ldots,\abar_r\in M^{\xbar}$ such that, for any $\bbar\in M^{\ybar}$, 
\[
\left|\mu(\phi(\xbar;\bbar))-\E(\abar_1,\ldots,\abar_r;\phi(\xbar;\bbar))\right|\leq\epsilon.
\]
In particular, if $\mu(\phi(\xbar;\bbar))>\epsilon$ then $\phi(\xbar;\bbar)$ is realized in $F(\phi,\epsilon)=\{\abar_1,\ldots,\abar_t\}$.
\end{corollary}
\begin{proof}
Fix $k$ and $\epsilon$ and let $r(k,\epsilon)$ be as in the VC-Theorem. Let $\phi(\xbar;\ybar,\zbar)$ be an $\cL$-formula, and let $\chi(\zbar)$ be an $\cL$-formula expressing that $\phi(\xbar;\ybar,\zbar)$ is $k$-NIP as a relation in $\xbar$ and $\ybar$. By the VC-theorem, if $A$ is a finite $\cL$-structure then
\[
A\models \forall \zbar\left(\chi(\zbar)\rightarrow \exists \xbar_1 \ldots \xbar_r\,\forall \ybar\,\left|\mu_\phi(\ybar,\zbar)-\E(\xbar_1,\ldots,\xbar_r;\phi(\xbar,\ybar,\zbar))\right|<\epsilon\right)
\]
(where the expression on the right is an $\cL^+$-sentence).
Therefore, by {\L}o\'{s}'s Theorem and elementarity, $M$ satisfies this sentence, which yields the desired result.
\end{proof}

\begin{corollary}\label{cor:fsg}
Let $\Delta=\{\phi_i(\xbar;\ybar_i):i\in I\}$ be a collection of NIP formulas. Then there is $M_0\preceq M$, of size at most $|I|+\aleph_0$, such that for any $i\in I$ and $\bbar\in M^{\ybar_i}$, if $\mu(\phi_i(\xbar;\bbar))>0$ then $\phi_i(\xbar;\bbar)$ is realized in $M$.
\end{corollary}
\begin{proof}
Let $M_0\preceq M$ be any model, of size at most $|I|+\aleph_0$, which contains the set $F(\phi_i,\epsilon)$ from Corollary \ref{cor:VCpf} for all $i\in I$ and rational $\epsilon>0$. 
\end{proof}

\begin{corollary}\label{cor:fsg2}
Suppose $M$ is pseudofinite, and fix an NIP formula $\phi(\xbar;\ybar)$. Then there is a countable set $A\subset M$ such that, for any closed $C\seq[0,1]$, the set
\[
\{\bbar\in M^{\ybar}:\mu(\phi(\xbar;\bbar))\in C\}
\]
is $\phi^{\opp}$-type-definable over $A$. 
\end{corollary}
\begin{proof}
Given $n>0$, we have $r_n\in\N$ and $\abar^n_1,\ldots,\abar^n_{r_n}\in M^{\xbar}$ such that, for all $\bbar\in M^{\ybar}$, $|\mu(\phi(\xbar,\bbar))-\E(\abar^n_1,\ldots,\abar^n_{r_n};\phi(\xbar;\bbar))|\leq \textstyle\frac{1}{n}$. Define 
\[
X_n=\left\{\bbar\in M^{\ybar}:d(\E(\abar^n_1,\ldots,\abar^n_{r_n};\phi(\xbar;\bbar)),C)\leq \textstyle\frac{1}{n}\right\}.
\]
Then, since $C$ is closed, it follows that $\{\bbar\in M^{\ybar}:\mu(\phi(\xbar;\bbar))\in C\}=\bigcap_{n>0}X_n$. So it suffices to show that each $X_n$ is $\phi^{\opp}$-definable over $A_n=\bigcup_{i=1}^{r_n}\abar^n_i$. Fix $n>0$ and, for $I\seq[r_n]$, define the formula
\[
\theta_I(\ybar):=\bigwedge_{i\in I}\phi(\abar^n_i;\ybar)\wedge\bigwedge_{i\in [r_n]\backslash I}\neg\phi(\abar^n_i;\ybar).
\]
Then $\theta_I(\ybar)$ is a $\phi^{\opp}$-formula over $A_n$. Set $\cF=\{I\seq[r_n]:d(\textstyle\frac{|I|}{n},C)\leq\textstyle\frac{1}{n}\}$. Then $X_n$ is defined by $\bigvee_{I\in\cF}\theta_I(\ybar)$.
\end{proof}

\subsection{NIP formulas and generic sets in pseudofinite groups}\label{sec:pre3}

\begin{definition}
Let $G$ be a group. Given $n\geq 1$, set $A\seq G$ is \textbf{left $n$-generic} (resp. \textbf{right $n$-generic}) if there are $n$ left translates (resp. right translates) of $A$ whose union is $G$. We say $A\seq G$ is \textbf{left generic} (resp. \textbf{right generic}) if it is left $n$-generic (resp. right $n$-generic) for some $n\geq 1$.
\end{definition}

We now assume that $\cL$ expands the language of groups, and we let $G$ be a fixed, sufficiently saturated $\cL$-structure which is an elementary extension of an ultraproduct of finite groups. Note that the pseudofinite counting measure $\mu$ on $G$ is left and right invariant. 

\begin{definition}
Let $\phi(x)$ be a formula.
\begin{enumerate}
\item Let $\phi^{\ell}(x;y)$ denote the formula $\phi(y\cdot x)$.
\item Let $\phi^r(x;y)$ denote the formula $\phi(x\cdot y)$.
\end{enumerate}
\end{definition}

Given a formula $\phi(x)$, note that $\phi^r(x;y)=(\phi^{\ell})^{\opp}(x;y)$. In particular, $\phi^{\ell}(x;y)$ is NIP if and only if $\phi^r(x;y)$ is NIP.

\begin{corollary}\label{cor:VCgen}
For any $k\geq 1$ and $\epsilon>0$ there is $n=n(k,\epsilon)$ such that, for any formula $\phi(x)$, if $\phi^{\ell}(x; y)$ is $k$-NIP and $\mu(\phi(x))>\epsilon$, then $\phi(x)$ is left $n$-generic and right $n$-generic.
\end{corollary}
\begin{proof}
Fix $k\geq 1$ and $\epsilon>0$. Let $n=\max\{r(k,\epsilon),r(2^k,\epsilon)\}$ be given by Corollary \ref{cor:VCpf}. Suppose $\phi(x)$ is a formula such that $\phi^{\ell}(x; y)$ is $k$-NIP, and assume $\mu(\phi(x))>\epsilon$. Then $\mu(\phi(bx ))>\epsilon$ for any $b\in G$ by invariance of $\mu$. By Corollary \ref{cor:VCpf}, there is $F\subset G$, of size at most $n$, such that $\phi(bx)$ is realized in $F$ for any $b\in G$. So the right translates of $\phi(x)$ by elements in $F\inv$ cover $G$, i.e. $\phi(x)$ is right $n$-generic. By choice of $n$ and the same argument applied to $\phi^r(x;y)$ (which is $2^k$-NIP), we see that $\phi(x)$ is left $n$-generic.
\end{proof}

\begin{corollary}\label{cor:left=right}
Let $\phi(x)$ be a formula such that $\phi^{\ell}(x;y)$ is NIP. The following are equivalent:
\begin{enumerate}[$(i)$]
\item $\phi(x)$ is left generic;
\item $\phi(x)$ is right generic;
\item $\mu(\phi(x))>0$.
\end{enumerate}
\end{corollary}
\begin{proof}
$(i)\Rightarrow(iii)$ and $(ii)\Rightarrow(iii)$ are by invariance and finite additivity of $\mu$. $(iii)\Rightarrow (i)$ and $(iii)\Rightarrow (ii)$ are by Corollary \ref{cor:VCgen}.
\end{proof}

In light of the previous corollary, we will just say $\phi(x)$ is \emph{generic} (or \emph{$n$-generic}), in the case that $\phi^{\ell}(x;y)$ is NIP and $\mu(\phi(x))>0$.

\begin{corollary}\label{cor:genideal}
Let $\phi(x)$ be a formula such that $\phi^{\ell}(x;y)$ is NIP. Then at least one of $\phi(x)$ or $\neg\phi(x)$ is generic.
\end{corollary}
\begin{proof}
At least one of $\phi(x)$ or $\neg\phi(x)$ must have positive $\mu$-measure.
\end{proof}

The final goal of this section is to show that, with respect to NIP formulas, the pseudofinite counting measure is the unique left-invariant finitely additive probability measure (this is made precise in the following theorem). Let $\Def(G)$ denote the Boolean algebra of all formulas in one free variable. 

\begin{theorem}\label{thm:unique}
Suppose $\cB$ is a left-invariant sub-algebra of $\Def(G)$ such that, for any $\phi(x)\in\cB$, $\phi^\ell(x;y)$ is NIP. Then the restriction of $\mu$ to $\cB$ is the unique left-invariant finitely additive probability measure on $\cB$.
\end{theorem}
\begin{proof}
Let $\nu$ be a left-invariant finitely additive probability measure on $\cB$ and fix $\phi(x)\in\cB$. We show that $\nu(\phi(x))=\mu(\phi(x))$. Let $\delta(x;y)$ denote $\phi(x\cdot y)$. Note that $\delta(x;y)$ is NIP and $\delta(a;y)\in\cB$ for any $a\in G$. Now fix $\epsilon>0$.   By Corollary \ref{cor:VCpf}, there are (not necessarily distinct) $a_1,\ldots,a_r\in G$ such that, for any $b\in G$, 
\[
|\mu(\delta(x;b))-\E(a_1,\ldots,a_r;\delta(x;b))|\leq\epsilon.
\]
Let $M_0\prec G$ be a small model containing $a_1,\ldots,a_r$ and any parameters in $\phi(x)$. By Proposition \ref{prop:VCmt}, there are (not necessarily distinct) $b_1,\ldots,b_s\in G$ such that, for any $a\in M_0$, 
\[
|\nu(\delta(a;y))-\E(b_1,\ldots,b_s;\delta(a;y))|\leq\epsilon.
\]
Note that $\nu(\delta(a;y))=\nu(\phi(x))$ for any $a\in G$ by left-invariance of $\nu$, and $\mu(\delta(x;b))=\mu(\phi(x))$ for any $b\in G$ by right-invariance of $\mu$. For $1\leq i\leq r$ and $1\leq j\leq s$, set $\delta_{i,j}=1$ if $G\models\delta(a_i,b_j)$, and set $\delta_{i,j}=0$ otherwise. Given real numbers $u,v$, we write $u\approx_\epsilon v$ to denote $|u-v|\leq\epsilon$. We have
\begin{multline*}
\textstyle\nu(\phi(x))= \frac{1}{r}\sum\limits_{i=1}^r\nu(\delta(a_i,y))\approx_\epsilon\frac{1}{r}\sum\limits_{i=1}^r\E(\bbar;\delta(a_i,y))=\frac{1}{rs}\sum\limits_{i=1}^r\sum\limits_{j=1}^s\delta_{i,j}\\
\textstyle =\frac{1}{rs}\sum\limits_{j=1}^s\sum\limits_{i=1}^r\delta_{i,j}=\frac{1}{s}\sum\limits_{j=1}^s\E(\abar;\delta(x,b_j))\approx_\epsilon\frac{1}{s}\sum\limits_{j=1}^s\mu(\delta(x;b_j))=\mu(\phi(x)),
\end{multline*}
and so $\nu(\phi(x))\approx_{2\epsilon}\mu(\phi(x))$. So $\mu(\phi(x))=\nu(\phi(x))$ since $\epsilon>0$ was arbitrary. 
\end{proof}

\section{Stabilizers and $G^{00}$}\label{sec:G00}

Throughout this section, and for the rest of the paper, we continue to work with a sufficiently saturated pseudofinite $\cL$-structure $G$ expanding a group.

\subsection{Stabilizers of formulas}

\begin{definition}
Let $\phi(x)$ be a formula.
\begin{enumerate}
\item Given $\epsilon\geq 0$, define
\[
\Stab^\epsilon_\mu(\phi(x))=\{g\in G:\mu(\phi(g\inv x)\smd\phi(x))\leq\epsilon\}.
\]
\item Define $\Stab_\mu(\phi(x))=\Stab^0_\mu(\phi(x))=\{g\in G:\mu(\phi(g\inv x)\smd\phi(x))=0\}$.
\end{enumerate}
\end{definition}

\begin{proposition}\label{prop:stabepdelta}
Suppose $\phi(x)$ is a formula such that $\phi^{\ell}(x;y)$ is NIP. Then, for any $\epsilon>0$, $\Stab^\epsilon_\mu(\phi(x))$ is left generic and $\phi^r$-type-definable over a countable parameter set.
\end{proposition}
\begin{proof}
Let $\psi(x;y_1,y_2)$ denote $\phi(y_1\cdot x)\smd\phi(y_2\cdot x)$, and note that $\psi(x;y_1,y_2)$ is NIP. By Corollary \ref{cor:VCpf}, we may fix a finite set $F\subset G$ such that, for any $b_1,b_2\in G$, if $\mu(\psi(x;b_1,b_2))>\epsilon$ then $\psi(x;b_1,b_2)$ is realized in $F$. Define an equivalence relation $\sim$ on $G$ such that $g\sim h$ if and only if $F\cap g\phi(G)=F\cap h\phi(G)$. Then $\sim$ has finitely many classes and so we may pick representatives $g_1,\ldots,g_n$. Let $X=\Stab^\epsilon_\mu(\phi(x))$. We show $G=g_1X\cup\ldots\cup g_n X$. Fix $h\in G$. Then $h\sim g_i$ for some $1\leq i\leq n$. It follows that $\psi(x;h\inv,g_i\inv)$ is not realized in $F$, and so
\[
\mu(\phi(h\inv g_ix)\smd\phi(x))= \mu(\psi(x;h\inv,g_i\inv))\leq\epsilon.
\] 
Therefore $g_i\inv h\in X$, and so $h\in g_iX$, as desired. 

Finally, let $\theta(x;y)$ denote $\phi(y\cdot x)\smd\phi(x)$, which is  NIP. We have 
\[
\Stab_\mu^\epsilon(\phi(x))=\{g\in G:\mu(\theta(x;g))\leq\epsilon\},
\] 
and so $\Stab^\epsilon_\mu(\phi(x))$ is $\theta^{\opp}$-type-definable over a countable parameter set by Corollary \ref{cor:fsg2}. Since any instance of $\theta^{\opp}(y;x)$ is equivalent to an instance of $\phi^r(x;y)$, we have the desired result.
\end{proof}

\begin{remark}
Let $\phi(x)$ and $\epsilon>0$ be as in the proof of Proposition \ref{prop:stabepdelta}. Note that if $\phi^{\ell}(x;y)$ is $k$-NIP, and $\pi$ denotes the shatter function for $(G,\{\phi(gx):g\in G\})$, then $\sim$ has at most $\pi(r(k,\epsilon))$ classes, where $r(k,\epsilon)$ is given by Corollary \ref{cor:VCpf}. By the Sauer-Shelah Lemma and Remark \ref{rem:VCbound}, $\Stab^\epsilon_\mu(\phi(x))$ is $n$-generic with  $n\leq \epsilon^{\nv O_k(1)}$.
\end{remark}

\begin{corollary}\label{cor:stabdelta}
Suppose $\phi(x)$ is a formula such that $\phi^{\ell}(x;y)$ is NIP. Then $\Stab_\mu(\phi(x))$ is a subgroup of $G$ of bounded index, which is $\phi^r$-type-definable over a countable parameter set. 
\end{corollary}
\begin{proof}
Using invariance and finite additivity of $\mu$, it is straightforward to check that $\Stab_\mu(\phi(x))$ is a subgroup of $G$. By definition, $\Stab_\mu(\phi(x))=\bigcap_{\epsilon\in\Q^+}\Stab^{\epsilon}_\mu(\phi(x))$. By Proposition \ref{prop:stabepdelta}, each set in this intersection is generic and $\phi^r$-type-definable over a countable parameter set. Therefore $\Stab_\mu(\phi(x))$ has bounded index and is $\phi^r$-type-definable over a countable parameter set. 
\end{proof}

Given a formula $\phi(x)$, the formula $\phi^{\ell}(x;y)$ is \emph{invariant} in the sense that any left translate of an instance of $\phi^{\ell}(x;y)$ is also an instance of $\phi^{\ell}(x;y)$. We want to work with the general class of formulas satisfying this property.

\begin{definition}
A formula $\delta(x;\ybar)$  is \textbf{(left) invariant} if, for any $a,\bbar\in G$, there is $\cbar\in G$ such that $\delta(ax;\bbar)$ is equivalent to $\delta(x;\cbar)$. 
\end{definition}

The main reason to work with invariant $\cL$-formulas is so that we have a well-defined action by $G$ on the space of $\delta$-types (defined below). However, given a formula $\delta(x;\ybar)$, which is invariant and NIP, it will be necessary to consider right translates of $\delta$-formulas in order to pinpoint type-definability at various steps of the subsequent work (as suggested by Proposition \ref{prop:stabepdelta}). Therefore, we set the following notation.

\begin{definition}
Given a formula $\delta(x;\ybar)$, let $\delta^r(x;\ybar,u)$ denote the formula $\delta(x\cdot u;y)$.
\end{definition}

Note that if an invariant $\cL$-formula $\delta(x;\ybar)$ is also \emph{right} invariant (e.g. if $G$ is abelian), then $\delta^r(x;\ybar,u)$ is essentially the same as $\delta(x;\ybar)$. However, in general, $\delta^r(x;\ybar,u)$ may behave quite differently. Most importantly, $\delta^r(x;\ybar)$ may be NIP, while $\delta^r(x;\ybar,u)$ is not, as demonstrated by the following example.

\begin{example}\label{ex:bad}
Given $k\in\N$, let $G_k$ be the group of permutations of $\{1,\ldots,k+1\}$, and let $H_k$ be the subgroup of permutations fixing $1$. Then, with $G_k$ as the ambient structure, the formula $yx\in H_k$ is $2$-stable (and thus $2$-NIP) since $H_k$ is a subgroup. But $yxy\in H_k$ is not $k$-NIP. To see this, let $X=\{2,\ldots,k+1\}$. Given $n\in X$ and $I\seq X$, let $a_n\in G_k$ be the transposition $(1~n)$, and let $b_I\in G_k$ be a permutation whose set of fixed points in $X$ is precisely $I$ (such a permutation always exists since $1\not\in X$). Then, given $n\in  X$ and $I\seq X$, $a_nb_Ia_n\in H_k$ if and only if $n\in I$. 

Now let $\cU$ be a nonprincipal ultrafilter on $\N$ and let $G=\prod_{\cU}G_k$. If $A=\prod_{\cU}H_k$, then $yx\in A$ is stable, while $yxy\in A$ has the independence property.  
\end{example}

Despite the behavior seen in the last example, we will still recover sufficiently good behavior for instances of the formula $\delta^r(x;\ybar,u)$ (see, e.g., Proposition \ref{prop:deltarfacts}).

The next goal is to define the local analog of $G^{00}$. We will first give an explicit construction using ``measure-stabilizers" of formulas, and then show that the object obtained behaves as expected (see Theorem \ref{thm:G00}).

\begin{definition}
Let $\delta(x;\ybar)$ be a formula. Define 
\[
G^*_\delta=\bigcap_{\abar\in G^{\ybar}}\Stab_\mu(\delta(x;\abar)).
\]
\end{definition}

\begin{lemma}\label{lem:boundedphi}
If $\delta(x;\ybar)$ is invariant and NIP, then there is a bounded set $A\seq G^{\ybar}$ such that $G^*_\delta=\bigcap_{\abar\in A}\Stab_\mu(\delta(x;\abar))$. 
\end{lemma}
\begin{proof}
Define an equivalence relation $\sim$ on $G^{\ybar}$ such that $\abar\sim \bbar$ if and only if $\mu(\delta(x;\abar)\smd\delta(x;\bbar))=0$. To find the desired set $A$, it suffices to show that $\abar\sim\bbar$ implies $\Stab_\mu(\delta(x;\abar))=\Stab_\mu(\delta(x;\bbar))$, and that $\sim$ has a bounded number of classes.

For the first claim, fix $\abar,\bbar,g\in G$. The formula $\delta(g\inv x;\abar)\smd\delta(x;\bbar)$ implies 
\[
(\delta(g\inv x;\bbar)\smd \delta(g\inv x;\abar))\vee (\delta(g\inv x;\abar)\smd \delta(x;\abar))\vee(\delta(x;\abar)\smd\delta(x;\bbar)).
\]
So if $\abar\sim \bbar$ and $g\in \Stab_\mu(\delta(x;\abar))$, then $g\in\Stab_\mu(\delta(x;\bbar))$ by invariance and finite additivity of $\mu$. 

The second claim is standard fact about NIP formulas (details are included for the sake of clarity). If $\sim$ has unboundedly many classes then by Erd\H{o}s-Rado there is an indiscernible sequence $(\bbar_i)_{i<\omega}$, and some $\epsilon>0$, such that $\mu(\delta(x;\bbar_i)\smd\delta(x;\bbar_j))\geq\epsilon$ for all $i\neq j$. Then $\{\delta(x;\bbar_{2i})\smd\delta(x;\bbar_{2i+1}):i<\omega\}$ is consistent by \cite[Lemma 2.8]{HPP}. This contradicts that $\delta(x;\ybar)$ is NIP and thus has finite alternation number (e.g. \cite[Theorem 12.17]{pobook}).
\end{proof}

From Corollary \ref{cor:stabdelta} and Lemma \ref{lem:boundedphi}, we immediately obtain the following result.

\begin{corollary}\label{cor:G00}
If $\delta(x;\ybar)$ is invariant and NIP then $G^*_\delta$ is a $\delta^r$-type-definable bounded-index subgroup of $G$.
\end{corollary}

\subsection{Stabilizers of types}

\begin{definition}
Fix an invariant formula $\delta(x;\ybar)$.
\begin{enumerate}
\item Given $A\seq G$, let $S_\delta(A)$ denote the space of \textbf{complete $\delta$-types} (i.e. maximal consistent sets of instances of $\delta$) with parameters from $A$. 
\item Given $p\in S_\delta(G)$, let $\Stab(p)=\{g\in G:gp=p\}$ (where $gp=\{\phi(g\inv x):\phi(x)\in p\}$).
\item A $\delta$-type $p$ is \textbf{left generic} (resp. \textbf{right generic}) if every formula in $p$ is left generic (resp. right generic).
\end{enumerate}
\end{definition}

\begin{proposition}\label{prop:deltarfacts}
Suppose $\delta(x;\ybar)$ is invariant and NIP.
\begin{enumerate}[$(a)$]
\item If $\phi(x)$ is a $\delta^r$-formula then $\phi^{\ell}(x;y)$ is NIP.
\item Given $p\in S_{\delta^r}(G)$, the following are equivalent:
\begin{enumerate}[$(i)$]
\item $p$ is left generic;
\item $p$ is right generic;
\item $\mu(\phi(x))>0$ for all $\phi(x)\in p$.
\end{enumerate}
\item The space of (left) generic types in $S_{\delta^r}(G)$ is nonempty and invariant under left and right multiplication.

\end{enumerate}
\end{proposition}
\begin{proof}
Part $(a)$. Fix $k\geq 1$ such that $\delta(x;\ybar)$ is $k$-NIP. We first claim that, for any $\bbar,c\in G$, if $\phi(x)$ denotes $\delta^r(x;\bbar,c)$, then $\phi^{\ell}(x;y)$ is $k$-NIP. To see this, suppose we have $(r_i)_{i\in [n]}$ and $(s_I)_{I\seq [n]}$ such that $\delta^r(s_Ir_i;\bbar,c)$ holds if and only if $i\in I$. For any $I$, there is $\abar_I$ such that $\delta(s_I\cdot x;\bbar)$ is equivalent to $\delta(x;\abar_i)$. So, setting $g_i=r_ic$, we have $\delta(g_i;\abar_I)$ if and only if $i\in I$. So $n<k$. Part $(a)$ now follows by induction on the construction of $\delta^r$-formulas.

Part $(b)$. This follows from part $(a)$ and Corollary \ref{cor:left=right}.

Part $(c)$. By finite additivity of $\mu$, the measure $0$ sets form an ideal, and so there are types $p\in S_{\delta^r}(G)$ satisfying condition $(iii)$ of part $(b)$. So the claims follow from parts $(a)$ and $(b)$.
\end{proof}

Given an  invariant NIP formula $\delta(x;\ybar)$ and a $\delta^r$-type $p$, we will just call $p$ \emph{generic} in case it is left generic (equivalently right generic).

The following are some technical observations that will be needed in the proof of Theorem \ref{thm:G00}.

\begin{proposition}\label{prop:G00tech}
Suppose $\delta(x;\ybar)$ is invariant.
\begin{enumerate}[$(a)$]
\item $G^*_{\delta}\seq\Stab_\mu(\phi(x))$ for any $\delta^r$-formula $\phi(x)$.
\item $G^*_{\delta}=G^*_{\delta^r}$.
\item For any $p\in S_{\delta^r}(G)$, there is a unique right coset $C$ of $G^*_\delta$ such that $p\models C$.
\end{enumerate}
\end{proposition}
\begin{proof}
Part $(a)$. Given $\delta^r$-formulas $\phi(x)$ and $\psi(x)$, we have $\Stab_\mu(\neg\phi(x))=\Stab_\mu(\phi(x))$ and $\Stab_\mu(\phi(x))\cap\Stab_\mu(\psi(x))\seq\Stab_\mu(\phi(x)\wedge\psi(x))$. So the claim follows by induction on the construction of $\delta^r$-formulas.

Part $(b)$. By definition, $G^*_{\delta^r}\seq G^*_\delta$. For the other containment, fix $g\in G^*_\delta$ and $\bbar,c\in G$. By right invariance of $\mu$, and since $g\in\Stab_\mu(\delta(x;\bbar))$, we have
\[
\mu(\delta^r(g\inv x;\bbar,c)\smd\delta^r(x;\bbar,c))=\mu(\delta(g\inv x;\bbar)\smd\delta(x;\bbar))=0,
\]
and so $g\in\Stab_\mu(\delta^r(x;\bbar,c))$. 

Part $(c)$. Note that all right cosets of $G^*_\delta$ are $\delta^r$-type-definable,  since $G^*_\delta$ is $\delta^r$-type-definable and $\delta^r$-formulas are right invariant. Since any complete $\delta^r$-type concentrates on at most one right coset of $G^*_\delta$, it suffices to show that every complete $\delta^r$-type concentrates on some right coset of $G^*_\delta$. Since $G^*_\delta$ is $\delta^r$-type-definable of bounded index, we may fix a small model $M\prec G$ such that all right cosets of $G^*_\delta$ are $\delta^r$-type-definable over $M$. Now, given $p\in S_{\delta^r}(G)$, if $a\in G$ realizes $p|_M$, then $p$ concentrates on $G^*_\delta a$.
\end{proof}

\begin{definition}
Fix a formula $\delta(x;\ybar)$.
\begin{enumerate}
\item Let $S^g_\delta(G)$ denote the set of generic $\delta$-types in $S_\delta(G)$.
\item Given $p\in S_{\delta}(G)$, define $\Stab(p)=\{g\in G:gp=p\}$.
\item Let $G^{00}_\delta$ denote the intersection of all $\delta$-type-definable bounded-index subgroups of $G$.
\end{enumerate}
\end{definition}

Note that, for any invariant formula $\delta(x;\ybar)$, the class of $\delta^r$-type-definable bounded-index subgroups of $G$ is closed under conjugation, and so $G^{00}_{\delta^r}$ is always a normal subgroup of $G$. The next theorem is the main result on $G^*_\delta$, for $\delta(x;\ybar)$ invariant and NIP.

\begin{theorem}\label{thm:G00}
Suppose $\delta(x;\ybar)$ is invariant and NIP.
\begin{enumerate}[$(a)$]

\item $G^*_\delta$ is a $\delta^r$-type-definable bounded-index subgroup of $G$.

\item $\displaystyle G^*_\delta=\bigcap_{p\in S^g_\delta(G)}\Stab(p)=\bigcap_{p\in S^g_\delta(G)}\bigcap_{\phi(x)\in p}\Stab_\mu(\phi(x))$.

\item If $p\in S_{\delta^r}(G)$ is generic then
\[
G^*_\delta=\Stab(p)=\bigcap_{\phi(x)\in p}\Stab_\mu(\phi(x)).
\]
\item $G^*_\delta=G^{00}_{\delta^r}$.
\item $G^{00}_{\delta^r}$ is defined by an intersection of countably many $\delta^r$-formulas.
\end{enumerate}
\end{theorem}
\begin{proof}
Part $(a)$. This is Corollary \ref{cor:G00}. 

Part $(b)$. We show
\[
G^*_\delta\seq \bigcap_{p\in S^g_\delta(G)}\bigcap_{\phi(x)\in p}\Stab_\mu(\phi(x))\seq \bigcap_{p\in S^g_\delta(G)}\Stab(p)\seq G^*_\delta.
\]
The first containment is immediate from Proposition \ref{prop:G00tech}$(a)$. For the second containment, we fix a generic type $p\in S_\delta(G)$ and show $\bigcap_{\phi(x)\in p}\Stab_\mu(\phi(x))\seq \Stab(p)$. Indeed, suppose $g\in \bigcap_{\phi(x)\in p}\Stab_\mu(\phi(x))$ and fix $\phi(x)\in p$. If $\phi(g\inv x)\not\in p$ then $\phi(g\inv x)\smd\phi(x)\in p$, which contradicts that $p$ is generic and $g\in\Stab_\mu(\phi(x))$. So $\phi(g\inv x)\in p$, and thus we have $g\in \Stab(p)$.

For the third containment, suppose $g\not\in G^*_\delta$. Then there is a $\delta$-formula $\phi(x)$ such that $\mu(\phi(g\inv x)\smd\phi(x))>0$, and so there is a generic type $p\in S_\delta(G)$ containing the formula $\phi(g\inv x)\smd\phi(x)$. So $g\not\in \Stab(p)$.  

Part $(c)$. Fix a generic type $p\in S_{\delta^r}(G)$. We show
\[
G^*_{\delta}\seq\bigcap_{\phi(x)\in p}\Stab_\mu(\phi(x))\seq\Stab(p)\seq G^*_{\delta}.
\]
The first containment is immediate from parts $(a)$ and $(b)$ of Proposition \ref{prop:G00tech}, and the second containment is similar to part $(b)$.  

For the third containment, first  fix $a\in M$ such that $p$ concentrates on $G^*_\delta a$ (such an $a$ exists by Proposition \ref{prop:G00tech}$(c)$). Now fix $g\in \Stab(p)$. Then $p\models g\inv G^*_\delta a$, and so $G^*_\delta a\cap g\inv G^*_\delta a$ is a consistent type, which is therefore realized in $G$. So there are $x,y\in G^*_\delta$ such that $xa=g\inv ya$, and so $g=yx\inv\in G^*_\delta$, as desired.

Part $(d)$. We have $G^{00}_{\delta^r}\seq G^*_{\delta}$ by part $(a)$. Now suppose $\Gamma$ is a $\delta^r$-type-definable subgroup of bounded index. We want to show $G^*_\delta\seq\Gamma$. Let $p\in S_{\delta^r}(G)$ be a generic $\delta^r$-type concentrating on $\Gamma$, and fix $a\in G^*_\delta$. Since $G^*_\delta=\Stab(p)$, it follows that $ap\models \Gamma$, and so $a\Gamma=\Gamma$. 

Part $(e)$. Without loss of generality, we may assume $\cL$ is countable (even finite) and $\delta(x;\ybar)$ is over $\emptyset$. By definition, $G^{00}_{\delta^r}$ is invariant over $\emptyset$. Since $G^{00}_{\delta^r}$ is $\delta^r$-type-definable, it must be type-definable over $\emptyset$, and thus is an intersection of countably many definable subsets of $G$. By an easy saturation exercise, one concludes that $G^{00}_{\delta^r}$ is defined by an intersection of countably many $\delta^r$-formulas. 
\end{proof}

We end this section by analyzing the situation when $\delta(x;\ybar)$ is stable.

\begin{definition}
Given a formula $\delta(x;\ybar)$, let $G^0_\delta$ denote the intersection of all $\delta$-definable finite-index subgroups of $G$.
\end{definition}

For stable $\delta(x;\ybar)$, the group $G^0_\delta$ is $\delta$-definable of finite index (this follows from \cite{HrPiGLF}, with further detail in  \cite{CPT}). The next corollary explains the relationship between $G^0_\delta$, $G^{00}_{\delta}$, $G^0_{\delta^r}$, and $G^{00}_{\delta^r}$ in this case (note that $\delta^r(x;\ybar,u)$ need not be stable, as demonstrated by Example \ref{ex:bad}).  

\begin{corollary}\label{cor:stable}
Assume $\delta(x;\ybar)$ is invariant and stable. Then $G^{00}_\delta=G^0_\delta$ and $G^{00}_{\delta^r}=G^0_{\delta^r}$. Moreover, $G^0_{\delta^r}$ is the normal core of $G^0_{\delta}$, and thus is $\delta^r$-definable of finite index. 
\end{corollary}

\begin{proof}
We first claim that, for any generic $\delta$-type $p\in S_\delta(G)$, if $p\models aG^{0}_\delta$ then $\Stab(p)= aG^{0}_\delta a\inv$. Indeed, fix $p\in S_\delta(G)$ generic and let $p\models aG^{0}_\delta$. If $g\in \Stab(p)$ then $p=gp\models gaG^{0}_\delta$, and so $gaG^{0}_\delta=aG^{0}_\delta$, i.e. $g\in aG^{0}_\delta a\inv$. Conversely, if $g\in aG^{0}_\delta a\inv$ then $ga G^{0}_\delta=aG^{0}_\delta$, and so $gp=p$ by \cite[Theorem 2.3$(ii)$]{CPT}. 

By the previous claim, and parts $(b)$ and $(d)$ of Theorem \ref{thm:G00},  we conclude that $G^{00}_{\delta^r}$ is the normal core of $G^0_{\delta}$, and therefore $G^{00}_{\delta^r}$ is $\delta^r$-definable of finite index. This further implies that $G^{00}_{\delta^r}=G^0_{\delta^r}$. 

It remains to show $G^{00}_\delta=G^0_\delta$. So suppose $H$ is a $\delta$-type-definable subgroup of $G$. Then $G^0_{\delta^r}=G^{00}_{\delta^r}\seq H$, and so $H$ is a union of cosets of $G^0_{\delta^r}$. Since $G^0_{\delta^r}$ has finite index, $H$ is definable. By compactness, $H$ is $\delta$-definable, and so $G^0_\delta\seq H$.
\end{proof}

\section{The local logic topology and $G^0$}\label{sec:G0}

Recall from \cite[Lemma 2.7]{PilCLG} that, if $\Gamma$ is a type-definable bounded-index subgroup of $G$, then we have the \emph{logic topology} on $G/\Gamma$ under which $G/\Gamma$ is a compact (Hausdorff) space, and a topological group when $\Gamma$ is normal. In particular, $X\seq G/\Gamma$ is closed if $\{x\in G:x\Gamma\in X\}$ is type-definable. In this section we show that if $\Gamma$ is $\delta$-type-definable, for some invariant $\delta(x,\ybar)$, then it suffices to consider $\delta$-type-definable sets in the construction of the logic topology on $G/\Gamma$. Certain aspects of this are likely in the folklore, and so some proofs will be brief.

\begin{lemma}\label{lem:haus}
Fix an invariant formula $\delta(x,\ybar)$ and suppose $\Gamma\leq G$ is $\delta$-type-definable of bounded index. Then, for any $\cL(G)$-formula $\phi(x)$, the set 
\[
X=\{a\in G:a\Gamma\cap\phi(G)\neq\emptyset\}
\]
is $\delta$-type-definable.
\end{lemma}
\begin{proof}
First, since $\Gamma$ is $\delta$-type-definable, it follows from saturation of $G$ that $X$ is type-definable (\emph{a priori}, using $\phi(x)$ and existential quantification over $\delta$-formulas). We need to show that $X$ is type-definable by $\delta$-formulas. By a \emph{$\delta\inv$-formula} we mean a formula of the form $\phi(x\inv)$ where $\phi(x)$ is a $\delta$-formula. A \emph{$\delta\inv$-type} is a small consistent set of $\delta\inv$-formulas. 

Note that $\Gamma$ is $\delta\inv$-type-definable since $\Gamma\inv=\Gamma$. Since $\Gamma$ has bounded index, we may fix a sequence $(p_i)_{i<\kappa}$ of left translates of $\delta\inv$-types, with $\kappa$ small, such that any coset of $\Gamma$ is definable by some $p_i$, and $p_0$ is a $\delta\inv$-type defining $\Gamma$. Let $A\subset G$ be a small set such that each $p_i$ is over $A$ and $\phi(x)$ is over $A$. Let $M\prec G$ be a small $|A|^+$-saturated model. Now let $S=\{p\in S_{\delta}(M):X\cap p(G)=\emptyset\}$. We show 
\begin{equation*}
X=\bigcap_{p\in S}\bigcup_{\psi(x)\in p}\neg\psi(G).\tag{$\dagger$}
\end{equation*}
By saturation and  type-definability of $X$, this will show that $X$ is $\delta$-type-definable.

By choice of $S$, the left-to-right containment in $(\dagger)$ is clear. So suppose $a\not\in X$ and let $p=\tp_{\delta}(a/M)$. It suffices to show $p\in S$. So suppose, toward a contradiction, that we have $b\in X\cap p(G)$. Then $b\Gamma\cap\phi(G)\neq\emptyset$. In particular, if $p_i$ is the type-definition of $b\Gamma$, then $p_i(x)\wedge\phi(x)$ is consistent, and thus realized by some $m\in M$. Then $p_0(b\inv m)$ holds and so, since $p_0$ is a $\delta\inv$-type, $m\inv\in M$, and $b\models p$, it follows that $p_0(a\inv m)$ holds. But then $m\in a\Gamma\cap\phi(G)$, contradicting that $a\not\in X$. 
\end{proof}

\begin{corollary}\label{cor:deltartop}
Fix an invariant formula $\delta(x;\ybar)$ and suppose $\Gamma\leq G$ is a $\delta$-type-definable of bounded index. Then $X\seq G/\Gamma$  is closed in the logic topology if and only if $\{a\in G:a\Gamma\in X\}$ is $\delta$-type-definable. 
\end{corollary}
\begin{proof}
Call $X\seq G/\Gamma$ \emph{$\delta$-closed} if $\{a\in G:a\Gamma\in X\}$ is $\delta$-type-definable. It suffices to show that the $\delta$-closed sets define a compact Hausdorff topology on $G/\Gamma$. Indeed, given this, since the logic topology clearly refines the $\delta$-topology, it will follow that the two topologies are the same. 

The verification that the $\delta$-closed sets generate a compact topology is exactly as in the usual case of the logic topology (\cite[Lemma 3.3]{LaPi} or \cite[Lemma 2.5]{PilCLG}). Moreover, Lemma \ref{lem:haus} is precisely what is needed to show that the standard argument of Hausdorff separation goes through. 
\end{proof}

Now, if $\delta(x;\ybar)$ is a formula and $\Gamma\leq G$ is $\delta$-definable of bounded index then, for any $p\in S_{\delta}(G)$ there is a unique left coset $C$ of $\Gamma$ such that $p\models C$. So we have a well-defined function $\pi_\Gamma\colon S_{\delta}(G)\to G/\Gamma$ such that $p\models \pi_\Gamma(p)$. The following conclusion is a straightforward from Corollary \ref{cor:deltartop}.

\begin{corollary}\label{cor:continuous}
Fix an invariant formula $\delta(x;\ybar)$ and suppose $\Gamma\leq G$ is $\delta$-type-definable of bounded index. Then $\pi_\Gamma$ is continuous.
\end{corollary}

\begin{remark}
Lemma \ref{lem:haus}, Corollary \ref{cor:deltartop}, and Corollary \ref{cor:continuous} hold for any (sufficiently saturated) $\cL$-structure $G$ expanding a group (i.e. $G$ need not be pseudofinite).
\end{remark}

As a special case of the above situation, we can work with $\Gamma=G^{00}_{\delta^r}$, where $\delta(x;\ybar)$ is invariant and NIP. We have already shown that, for $\delta(x,\ybar)$ invariant and NIP, $G^{00}_{\delta^r}$ behaves like the ``type-definable connected component" of $G$ localized at the formula $\delta^r$.
 Next, we show that $G^0_{\delta^r}$ fits into this picture the way one would expect from known facts in the global NIP setting. 

\begin{corollary}\label{cor:G00connect}
Suppose $\delta(x;\ybar)$ is invariant and NIP. Then $G^0_{\delta^r}$ is a normal $\delta^r$-type-definable subgroup of $G$ of bounded index, and $G^0_{\delta^r}/G^{00}_{\delta^r}$ is the connected component of the identity in $G/G^{00}_{\delta^r}$.
\end{corollary}
\begin{proof}
Let $C\seq G/G^{00}_{\delta^r}$ be the connected component of the identity, and recall that $C$ is a closed subgroup of $G/G^{00}_{\delta^r}$. Let $K$ be the pullback of $C$ to $G$. Then $K$ is a normal $\delta^r$-type-definable bounded-index subgroup of $G$ containing $G^{00}_{\delta^r}$. We also have $C=K/G^{00}_{\delta^r}$. Altogether, to prove the result, it suffices to show $K=G^0_{\delta^r}$.

We first show $K\seq G^0_{\delta^r}$. Let $H\leq G$ be $\delta^r$-definable of finite index. We have $G^{00}_{\delta^r}\leq H$, and $H/G^{00}_{\delta^r}$ is a clopen subgroup of $G/G^{00}_{\delta^r}$ containing the identity. So $C\seq H/G^{00}_{\delta^r}$, i.e., $K\seq H$.

Now, to prove $G^0_{\delta^r}\seq K$, fix $a\not\in K$. There is $X\seq G/G^{00}_{\delta^r}$ clopen such that $X=X\inv$, $aG^{00}_{\delta^r}\not\in X$, and $G^{00}_{\delta^r}\in X$. Let $A$ be the pullback of $X$ to $G$, and note that $a\not\in A$ and $A$ is definable. Let $H=\{g\in G:gA=A\}$. Then $H$ is a definable subgroup of $G$, and $G^{00}_{\delta^r}\seq H\seq A$. In particular, $H$ has finite index, and is $\delta^r$-definable by Corollary \ref{cor:deltartop}. Since $a\not\in H$, we have shown $a\not\in G^0_{\delta^r}$.
\end{proof}

\section{Generic compact domination}\label{sec:GCD}

Throughout this section, we fix an invariant NIP  formula $\delta(x;\ybar)$. 
 
 \begin{definition}$~$
 \begin{enumerate}
 \item Let $H_\delta$ denote the compact Hausdorff group $G/G^{00}_{\delta^r}$.
\item Let $\eta_\delta$ denote the normalized Haar measure on $H_\delta$.
\item Given $a\in G$, let $[a]_\delta$ denote the element $aG^{00}_{\delta^r}$ in $H_\delta$.
 \item  Let $\pi_\delta$ denote the map $\pi_{G^{00}_{\delta^r}}\colon S_{\delta^r}(G)\to H_\delta$ defined before Corollary \ref{cor:continuous}.
 \item Given $\alpha\in H_\delta$, define
 \[
 S^\alpha_{\delta^r}(G):=\pi_\delta\inv (\alpha)\cap S^g_{\delta^r}(G),
 \]
 i.e. $p\in S^\alpha_{\delta^r}(G)$ if and only if $p$ is a global generic $\delta^r$-type containing the $\delta^r$-type-definition of $\alpha$ (as a coset of $G^{00}_{\delta^r}$).
 \item  Given a $\delta^r$-formula $\phi(x)$, define 
\[
E_{\phi}=\{\alpha\in H_\delta: S^\alpha_{\delta^r}(G)\cap \phi(x)\neq\emptyset\text{ and }S^\alpha_{\delta^r}(G)\cap\neg\phi(x)\neq\emptyset\},
\]
where we identify a $\delta^r$-formula with a clopen set of types in $S_{\delta^r}(G)$.
 \item Given a $\delta^r$-formula $\phi(x)$ and a generic type $p\in S_{\delta^r}(G)$, define 
 \[
 U^p_{\phi}=\{[a]_\delta\in H_\delta: \phi(x)\in ap\}.
 \]
 \end{enumerate}
 \end{definition} 

Note that, in the previous definition, $U^p_{\phi}$ is well-defined since $G^{00}_{\delta^r}= \Stab(p)$ by Theorem \ref{thm:G00}$(c)$.

\begin{theorem}\label{thm:GCD}
For any $\delta^r$-formula $\phi(x)$, $E_{\phi}$ is closed and $\eta_\delta(E_{\phi})=0$.
\end{theorem}
\begin{proof}
The proof adapts parts of \cite{ChSi} and \cite{SimGCD}, and relies on the main results of \cite{SimRC} and \cite{SimGCD}. Throughout the proof we will use the fact that, for any $\delta^r$-formula $\phi(x)$, the formula $\phi(x\cdot y)$ is NIP (see Proposition \ref{prop:deltarfacts}$(a)$).  

First, we observe that $E_{\phi}=\pi_\delta(S^g_{\delta^r}(G)\cap \phi(x))\cap\pi_\delta(S^g_{\delta^r}(G)\cap\neg\phi(x))$. Thus $E_{\phi}$ is closed since $S^g_{\delta^r}(G)$, $\phi(x)$, and $\neg\phi(x)$ are closed, and $\pi_\delta$ is a continuous map between compact Hausdorff spaces. \medskip

\noindent\emph{Claim 1}: If $\phi(x)$ is a $\delta^r$-formula and $p\in S_{\delta^r}(G)$ is generic, then both $U^p_{\phi}$ and its complement are $F_\sigma$ subsets of $H_\delta$. 

\noindent\emph{Proof}: 
We are going to use \cite{SimRC}, which requires a countable theory. So let $T$ be the complete theory of $G$ in the language containing the group operation, $\phi(x)$, and constants for parameters in $\phi(x)$. For the rest of the proof we work in $T$, and view $\phi(x)$ as a formula with no parameters. Let $\theta(x;y_1,y_2)$ be the formula $\phi(y_1\cdot x)\wedge\neg\phi(y_2\cdot x)$, and note that $\theta(x;y_1,y_2)$ is NIP. By Corollary \ref{cor:fsg}, we may find a countable model $M\prec G$ such that $\theta(x;y_1,y_2)$ is over $M$ and, for any $a_1,a_2\in G$, if $\mu(\theta(x;a_1,a_2))>0$ then $\theta(x;a_1,a_2)$ is realized in $M$.

 Let $p_0$ be the global $\phi^\ell$-type obtained by restricting $p$ to instances of $\phi^\ell(x;y)$. Then $p_0$ is $M$-invariant. Indeed, if $a_1\equiv_M a_2$ and $\phi(a_1x)\wedge\neg\phi(a_2x)\in p_0$, then $\mu(\theta(x;a_1,a_2))>0$ since $p_0$ is generic, and so $\theta(x;a_1,a_2)$ is realized in $M$, a contradiction. By the main result of \cite{SimRC}, if $\Sigma=\{q\in S_y(M):\phi(ax)\in p_0 \text{ for } a\models q\}$ then both $\Sigma$ and its complement are $F_\sigma$  subsets of $S_y(M)$. Finally,
\[
\{a\in G:[a]_\delta\in U^p_{\phi}\}=\{a\in G:a\models q\text{ for some }q\in \Sigma\},
\]
and so we have the desired result by Corollary \ref{cor:continuous}. \claim{1}\medskip

By Claim 1, each generic $p\in S_{\delta^r}(G)$ induces a left-invariant finitely additive probability measure on $\delta^r$-formulas, by assigning the measure of a $\delta^r$-formula $\phi(x)$ to be $\eta_\delta(U^p_{\phi})$. (Therefore, by Theorem \ref{thm:unique}, we have $\eta_\delta(U^p_{\phi})=\mu(\phi)$ for any $\delta^r$-formula $\phi(x)$; however we will not need this for present proof.) For the rest of the proof, fix a generic type $p\in S_{\delta^r}(G)$ concentrating on $G^{00}_{\delta^r}$. \medskip

\noindent\emph{Claim 2}: For any $\delta^r$-formula $\phi(x)$ and any $\alpha\in U^p_{\phi}$, if $V\seq H_\delta$ is an open neighborhood of $\alpha$ then $\eta_\delta(U^p_{\phi}\cap V)>0$. 

\noindent\emph{Proof}: We follow the proof of Claim 2 of \cite[Theorem 4.2]{SimGCD}. Since $\pi_\delta\inv(\alpha)$ and $\pi_\delta\inv(\neg V)$ are disjoint closed subsets of $S_{\delta^r}(G)$, there is some $\delta^r$-formula $\psi(x)$ such that $\pi_\delta\inv(\alpha)\seq\psi(x)\seq\pi_\delta\inv(V)$. Fix $a\in \alpha$. Then $\psi(x)\in ap$, and so $\phi(x)\wedge\psi(x)\in ap$. Thus $\phi(x)\wedge\psi(x)$ is generic, and so 
\[
\eta_\delta(U^p_{\phi}\cap U^p_{\psi})=\eta(U^p_{\phi\wedge\psi})>0.
\]
Moreover, if $[g]_\delta\in U^p_{\psi}$ then $\psi(x)\in gp$, and so $[g]_\delta=\pi_\delta(gp)\in\pi(\psi(x))\seq V$. So $U^p_{\psi}\seq V$ and thus $\eta_\delta(U^p_{\phi}\cap V)>0$.\claim{2} \medskip

\noindent\emph{Claim 3}: For any $\delta^r$-formula $\phi(x)$ and generic type $p\in S_{\delta^r}(G)$, the set system $(H_\delta,\{\alpha U^p_{\phi}:\alpha\in H_\delta\})$ has finite VC-dimension.

\noindent\emph{Proof}:
This is similar to \cite[Lemma 3.19]{ChSi}. Suppose there are $\alpha_1,\ldots,\alpha_n\in H_\delta$ such that, for all $I\seq\{1,\ldots,n\}$ there is $\beta_I\in H_\delta$ with $\alpha_i\in \beta_IU^p_{\phi}$ if and only if $i\in I$. Let $\alpha_i=[a_i]_{\delta}$ and $\beta_I=[b_I]_\delta$, for $a_i,b_I\in G$. Then $[b_Ia_i\inv]_\delta\in U^p_{\phi}$ if and only if $i\in I$, i.e., $\phi(b_I\inv a_ix)\in p$ if and only if $i\in I$. Let $c$ realize the restriction of $p$ to the parameters $\{a_i\}_{i\leq n}\cup\{b_I\}_{I\seq\{1,\ldots,n\}}$. Then $G\models\phi^\ell(a_ic;b_I\inv)$ if and only if $i\in I$. So the result follows since $\phi^\ell(x;y)$ is NIP. \claim{3}\medskip

Recall that $H_\delta$ is second countable by Theorem \ref{thm:G00}$(e)$. 
Applying Claims 1, 2, and 3, and \cite[Theorem 3.7]{SimGCD}, we have $\eta_\delta(\partial U^p_{\phi})=0$. So, to prove the result, it suffices to show $E_{\phi}\seq\partial U^p_{\phi}$. The argument follows the proof of \cite[Theorem 5.3]{ChSi}. We first need one final claim.\medskip

\noindent\emph{Claim 4}: If $p\in S_{\delta^r}(G)$ is generic then $S^g_{\delta^r}(G)=\overline{\{gp:g\in G\}}$.

\noindent\emph{Proof}: (See also \cite{NewTD}.) Fix a generic type $p\in S_{\delta^r}(G)$. Note that $S^g_{\delta^r}(G)$ is closed, and clearly contains $gp$ for any $g\in G$. For the other containment, suppose $q\in S_{\delta^r}(G)$ is generic, and let $\psi(x)\in q$. We want to find $g\in G$ such that $\psi(x)\in gp$. Since $\psi(x)$ is generic there are $g_1,\ldots,g_n\in G$ such that $G=g_1\psi(G)\cup\ldots\cup g_n\psi(G)$, and so $\psi(x)\in g\inv_ip$ for some $i$.\claim{4}\medskip

Now fix $\alpha\in E_{\phi}$. Let $V\seq H_\delta$ be open, with $\alpha\in V$. Since $\alpha\in E_{\phi}$, there are $q,q'\in S^\alpha_{\delta^r}(G)$ such that $\phi(x)\in q$ and $\neg\phi(x)\in q'$. Let $S=\pi_\delta\inv(V)$, and note that $S\seq S_{\delta^r}(G)$ is open, with $\pi_\delta\inv(\alpha)\seq S$. In particular, $q\in S\cap \phi(x)$ and $q'\in S\cap\neg\phi(x)$, and so these are are nonempty open sets in $S_{\delta^r}(G)$. By Claim 4, there are $g,g'\in G$ such that $gp\in S\cap\phi(x)$ and $g'p\in S\cap\neg \phi(x)$. Since $p$ concentrates on $G^{00}_{\delta^r}$, we have $\pi_\delta(gp)\in V\cap U^p_{\phi}$ and $\pi_\delta(g'p)\in V\cap\neg U^p_{\phi}$. Altogether, $\alpha\in \partial U^p_{\phi}$. 
\end{proof}

Generic compact domination can be used to derive stronger uniqueness statement for Keisler measures on $\delta^r$-formulas. In particular, we say that a finitely additive probability measure $\nu$ on $\delta^r$-formulas \emph{lifts $\eta_\delta$} if $\eta_\delta(W)=\nu(\pi_\delta\inv(W))$ for any Borel set $W\seq H_\delta$ (as usual, we identify $\nu$ with a regular Borel probability measure on $S_{\delta^r}(G)$). For the rest of this section, we view the pseudofinite counting measure $\mu$ as restricted to $\delta^r$-formulas. Using left-invariance, one sees  that $\mu$ lifts $\eta_\delta$.

\begin{corollary}\label{cor:GCD}
If $\nu$ is a left $G^{00}_{\delta^r}$-invariant Keisler measure on $\delta^r$-formulas, and $\nu$ lifts $\eta_\delta$, then $\nu=\mu$. 
\end{corollary}
\begin{proof}
In addition to Theorem \ref{thm:GCD}, the other main ingredient of the proof is that $\nu(\phi(x))=0$ for any non-generic $\delta^r$-formula $\phi(x)$. So let us first assume this is true, and prove the result. Under this assumption, we have $\nu(W)=\nu(S^g_{\delta^r}(G)\cap W)$ for any Borel $W\seq S_{\delta^r}(G)$ (as usual, we view $\nu$ as a regular Borel probability measure on $S_{\delta^r}(G)$). Given a $\delta^r$-formula $\phi(x)$, define $S_\phi=S^g_{\delta^r}\cap\phi(x)$ and $Z_\phi=\pi_\delta\inv(\pi_\delta(S_\phi))\backslash S_\phi$. One easily checks that $\pi_\delta(S^g_{\delta^r}(G)\cap Z_\phi)\seq E_\phi$ and so, since $\nu$ lifts $\eta_\delta$, we have $\nu(Z_\phi)=\nu(S^g_{\delta^r}(G)\cap Z_\phi)\leq \eta_\delta(E_\phi)=0$. Therefore
\[
\nu(\phi(x))=\nu(S_\phi)+\nu(Z_\phi)=\nu(\pi_\delta\inv(\pi_\delta(S_\phi))=\eta_\delta(\pi_\delta(S_\phi)).
\]
So $\nu$ is uniquely determined by $\eta_\delta$. 

To finish the proof, we fix a $\delta^r$-formula $\phi(x)$, with $\nu(\phi(x))>0$, and show that $\phi(x)$ is generic. This argument is extracted from results in \cite{ChSi}. First, we claim that $\phi(x)$ \emph{does not $G$-divide}, i.e., there is no integer $k\geq 1$ and sequence $(b_i)_{i<\omega}$ such that $\{\phi^\ell(x;b_i):i<\omega\}$ is $k$-inconsistent (see \cite[Definition 3.2]{ChSi}). Indeed, if $\phi(x)$ does $G$-divide then for any bounded cardinal $\lambda$, we can find an integer $k\geq 1$ and a sequence $(b_i)_{i<\lambda}$ such that $\{\phi^\ell(x;b_i):i<\lambda\}$ is $k$-inconsistent. Choose $\lambda$ large enough so that, by Erd\"{o}s-Rado, any $[G:G^{00}_{\delta^r}]$-coloring of pairs from $\lambda$ admits an infinite monochromatic subset. After passing to a subsequence, we obtain $(b_i)_{i<\omega}$ and $a\in G$ such that $p:=\{\phi^\ell(x;b_i):i<\omega\}$ is $k$-inconsistent and $ab_ib_j\inv\in G^{00}_{\delta^r}$ for all $i<j$. Now define the  measure $\nu_0$ such that $\nu_0(\psi(x))=\nu(\psi(b_0\inv a\inv x))$ for any $\delta^r$-formula $\psi(x)$. Since $\nu$ is $G^{00}_{\delta^r}$-invariant, we have $\nu_0(\phi^\ell(x;b_i))=\nu(\phi(x))>0$ for all $i>0$. This contradicts $k$-inconsistency of $p$ by a standard exercise on finitely additive probability measures. 

We now apply   \cite[Proposition 3.33]{ChSi} and \cite{NewTD} (see also \cite[Fact 3.29]{ChSi}) to conclude that $\phi(x)$ is \emph{weakly generic}, i.e., there are $g_1,\ldots,g_n\in G$ such that, if $\theta(x):=\bigvee_{i=1}^n\phi(g_ix)$, then $\neg\theta(x)$ is not generic (\cite[Proposition 3.33]{ChSi} is stated for NIP theories, but the proof only uses that $\phi^\ell(x;y)$ is NIP). So $\theta(x)$ is generic by Corollary \ref{cor:genideal}, which implies $\phi(x)$ is generic, as desired. 
\end{proof}
 
 Since any left-invariant Keisler measure on $\delta^r$-formulas lifts $\eta_\delta$, the previous corollary gives a second proof that $\mu$ is the unique such measure (which also follows from Theorem \ref{thm:unique} and Proposition \ref{prop:deltarfacts}$(a)$). 

\begin{remark}
Let $\nu$ be a Keisler measure on $\delta^r$-formulas. Then the following are equivalent.
\begin{enumerate}[$(i)$]
\item $\nu$ is left $G^{00}_{\delta^r}$-invariant.
\item $\nu(\phi(gx)\smd\phi(x))=0$ for any $\delta^r$-formula $\phi(x)$ and $g\in G^{00}_{\delta^r}$.
\item $\nu(\phi(x))>0$ for any generic $\delta^r$-formula $\phi(x)$. 
\end{enumerate}
Indeed $(ii)\Rightarrow(i)$ is clear, $(i)\Rightarrow(iii)$ was shown in Corollary \ref{cor:GCD}, and $(iii)\Rightarrow(ii)$ follows from Theorem \ref{thm:G00}$(d)$.
\end{remark}

\bibliographystyle{amsplain}

\begin{thebibliography}{10}

\bibitem{ChSi}
Artem Chernikov and Pierre Simon, \emph{Definably amenable {NIP} groups}, J.
  Amer. Math. Soc. \textbf{31} (2018), no.~3, 609--641. \MR{3787403}

\bibitem{CPT}
G.~Conant, A.~Pillay, and C.~Terry, \emph{A group version of stable
  regularity}, Math. Proc. Cambridge Philos. Soc. \textbf{168} (2020), no.~2,
  405--413. \MR{4064112}
  
\bibitem{CPTNIP}
\bysame, \emph{Structure and regularity for subsets of groups with finite VC-dimension}, arXiv 1802.04246, 2018.

\bibitem{GreenSLAG}
B.~Green, \emph{A {S}zemer\'edi-type regularity lemma in abelian groups, with
  applications}, Geom. Funct. Anal. \textbf{15} (2005), no.~2, 340--376.
  \MR{2153903}

\bibitem{HaWe}
David Haussler and Emo Welzl, \emph{{$\epsilon$}-nets and simplex range
  queries}, Discrete Comput. Geom. \textbf{2} (1987), no.~2, 127--151.
  \MR{884223}

\bibitem{HPP}
Ehud Hrushovski, Ya'acov Peterzil, and Anand Pillay, \emph{Groups, measures,
  and the {NIP}}, J. Amer. Math. Soc. \textbf{21} (2008), no.~2, 563--596.
 

\bibitem{HrPiGLF}
Ehud Hrushovski and Anand Pillay, \emph{Groups definable in local fields and
  pseudo-finite fields}, Israel J. Math. \textbf{85} (1994), no.~1-3, 203--262.
 

\bibitem{HP}
\bysame, \emph{On {NIP} and invariant measures}, J. Eur. Math. Soc. (JEMS)
  \textbf{13} (2011), no.~4, 1005--1061.

\bibitem{HPS}
Ehud Hrushovski, Anand Pillay, and Pierre Simon, \emph{Generically stable and
  smooth measures in {NIP} theories}, Trans. Amer. Math. Soc. \textbf{365}
  (2013), no.~5, 2341--2366. 

\bibitem{KoPaWo}
J\'anos Koml\'os, J\'anos Pach, and Gerhard Woeginger, \emph{Almost tight
  bounds for {$\epsilon$}-nets}, Discrete Comput. Geom. \textbf{7} (1992),
  no.~2, 163--173. 

\bibitem{LaPi}
D.~Lascar and A.~Pillay, \emph{Hyperimaginaries and automorphism groups}, J.
  Symbolic Logic \textbf{66} (2001), no.~1, 127--143. 

\bibitem{NewTD}
Ludomir Newelski, \emph{Topological dynamics of definable group actions}, J.
  Symbolic Logic \textbf{74} (2009), no.~1, 50--72. 

\bibitem{PilCLG}
Anand Pillay, \emph{Type-definability, compact {L}ie groups, and o-minimality},
  J. Math. Log. \textbf{4} (2004), no.~2, 147--162. 

\bibitem{pobook}
Bruno Poizat, \emph{A course in model theory}, Universitext, Springer-Verlag,
  New York, 2000, An introduction to contemporary mathematical logic,
  Translated from the French by Moses Klein and revised by the author.
  \MR{1757487 (2001a:03072)}
  
  \bibitem{Sibook}
Pierre Simon, \emph{A guide to {NIP} theories}, Lecture Notes in Logic,
  vol.~44, Association for Symbolic Logic, Chicago, IL; Cambridge Scientific
  Publishers, Cambridge, 2015.

\bibitem{SimRC}
Pierre Simon, \emph{Rosenthal compacta and {NIP} formulas}, Fund. Math.
  \textbf{231} (2015), no.~1, 81--92. 

\bibitem{SimGCD}
\bysame, \emph{V{C}-sets and generic compact domination}, Israel J. Math.
  \textbf{218} (2017), no.~1, 27--41. 

\end{thebibliography}
\end{document}